%% file: heavy-covariance-eprint.tex
\documentclass[11pt,a4paper]{article}

\usepackage{fullpage}
\usepackage{microtype}
\usepackage{xfrac}
\usepackage{enumitem}

\usepackage{color}
\usepackage{algorithm,algorithmic}
\usepackage{amsmath,amsthm,amssymb,amsfonts}
\usepackage[numbers]{natbib}

\usepackage[dvipsnames]{xcolor}
\usepackage[colorlinks]{hyperref}
\hypersetup{
    linkcolor=red,
    citecolor=OliveGreen,
    filecolor=black,
    urlcolor=black,
}

\usepackage{cleveref}

\title{Affine Invariant Covariance Estimation \\ for Heavy-Tailed Distributions}

\author{
Dmitrii M. Ostrovskii~\thanks{Viterbi School of Engineering, University of Southern California, Los Angeles, USA.}
\thanks{This work has been done while the first author was at the SIERRA Project-Team, of Inria, Paris, France.}
\\ \texttt{dostrovs@usc.edu}
\and 
Alessandro Rudi~\thanks{SIERRA Project-Team, Inria and \' Ecole Normale Sup\' erieure, PSL Research University, Paris, France.}\\ \texttt{alessandro.rudi@inria.fr}
}

\input{macro}

\begin{document}

\maketitle

\input{abstract}

\input{intro}
\input{background}
\input{algorithm}
\input{theory}
\input{applications}
\input{conclusion} 
\input{acks}

\newpage
\appendix

\input{df-lemma}
\input{proof-stability}
\input{proof-lepski}
\input{proof-ridge}

\bibliography{biblio}
\bibliographystyle{alpha}

\end{document}

%% file: macro.tex
\newcommand{\Nystrom}[1]{{Nystr\"om}}

\providecommand{\scal}[2]{\left\langle{#1},{#2}\right\rangle}

\providecommand{\tr}{\operatorname{Tr}}

\newcommand{\R}{\mathbb R}

\newcommand{\la}{\lambda}
\newcommand{\eps}{\epsilon}

\renewcommand{\S}{{S}}

\renewcommand{\k}{{K}}

\def\k{\kappa}

\def\Id{\mathbf{I}}

\def\cH{\mathcal{H}}

\newcommand{\A}{\mathbf{A}}
\newcommand{\bB}{\mathbf{B}}

\newcommand{\tO}{\wt{O}}
\newcommand{\WM}{\textup{WM}}
\newcommand{\lmin}{\lambda_{\min}}

\renewcommand{\S}{{\mathbf S}}
\newcommand{\J}{{\mathbf J}}
\newcommand{\lam}{\lambda}

\newcommand{\wt}{\widetilde}

\newcommand{\lang}{\langle}
\newcommand{\rang}{\rangle}

\newcommand{\E}{\mathds{E}}

\newcommand{\wh}{\widehat}

\newcommand{\df}{\mathsf{df}}

\newcommand{\prccq}{\preccurlyeq}
\newcommand{\succq}{\succcurlyeq}

\newcommand{\deff}{d_{\text{eff}}}
\newcommand{\reff}{\mathtt{r}}
\newcommand{\veps}{\varepsilon}

\newcommand{\vphi}{\varphi}

\newcommand{\bT}{\mathbf{T}}
\newcommand{\bJ}{\mathbf{J}}

\usepackage{comment}
\usepackage{dsfont}
\renewcommand{\le}{\leqs}
\renewcommand{\ge}{\geqs}
\newcommand{\leqs}{\leqslant}
\newcommand{\geqs}{\geqslant}
\newcommand{\bLambda}{\boldsymbol{\Lambda}}

\newcommand{\cJ}{\mathcal{J}}
\newcommand{\Z}{\mathds{Z}}
\newcommand{\thetamin}{\theta_{\min}}
\newcommand{\thetamax}{\theta_{\max}}
\newcommand{\whj}{\wh \jmath}
\newcommand{\proofpoint}[1]{$\boldsymbol{{#1}^o}.$}

\newcommand{\kresp}{\varkappa}
\renewcommand{\le}{\leqs}
\renewcommand{\ge}{\geqs}
\renewcommand{\leq}{\leqs}
\renewcommand{\geq}{\geqs}
\renewcommand{\preceq}{\prccq}

\renewcommand{\epsilon}{\veps}
\newcommand{\cond}{\textrm{cond}}
\newcommand{\bR}{\mathbf{R}}
\newcommand{\q}{\mathtt{q}}

\newtheorem{example}{Example}[section]
\newtheorem{theorem}{Theorem}[section]
\newtheorem{lemma}{Lemma}[section]

\newtheorem{corollary}{Corollary}[section]

\newtheorem{proposition}{Proposition}[section]
\newtheorem{remark}{Remark}[section]

\newtheorem{assumption}{Assumption}[section]

\newtheorem{definition}{Definition}[section]

\newcommand{\eqal}[1]{\begin{align}#1\end{align}}
\newcommand{\bpr}{\begin{proof}}
\newcommand{\epr}{\end{proof}}
\newcommand{\be}{\begin{equation}}
\newcommand{\ee}{\end{equation}}

\newcommand{\bd}{\begin{definition}}
\newcommand{\ed}{\end{definition}}

\newcommand{\bi}{\begin{itemize}}
\newcommand{\ei}{\end{itemize}}

\newcommand{\ba}{\begin{assumption}}
\newcommand{\ea}{\end{assumption}}

\newcommand{\bre}{\begin{restatable}}
\newcommand{\ere}{\end{restatable}}

\newcommand{\br}{\begin{remark}}
\newcommand{\er}{\end{remark}}

\newcommand{\bp}{\begin{proposition}}
\newcommand{\ep}{\end{proposition}}

\newcommand{\blm}{\begin{lemma}}
\newcommand{\elm}{\end{lemma}}

\newcommand{\bt}{\begin{theorem}}
\newcommand{\et}{\end{theorem}}

\newcommand{\bcor}{\begin{corollary}}
\newcommand{\ecor}{\end{corollary}}

\newcommand{\bex}{\begin{example}}
\newcommand{\eex}{\end{example}}

\crefname{assumption}{Assumption}{Assumptions}
\crefname{equation}{Eq.}{Eqs.}
\crefname{figure}{Fig.}{Figs.}
\crefname{table}{Table}{Tables}
\crefname{section}{Sec.}{Secs.}
\crefname{theorem}{Thm.}{Thms.}
\crefname{lemma}{Lemma}{Lemmas}
\crefname{corollary}{Cor.}{Cors.}
\crefname{example}{Example}{Examples}
\crefname{appendix}{Appendix}{Appendixes}
\crefname{remark}{Remark}{Remark}

\newcommand{\Shat}{\widehat{\S}}
\newcommand{\Stilde}{\widetilde{\S}}
\def\rank{\operatorname{rank}}

%% file: abstract.tex
\begin{abstract}
In this work we provide an estimator for the covariance matrix of a heavy-tailed multivariate distribution.
We prove that the proposed estimator $\widehat{\mathbf{S}}$ admits an \textit{affine-invariant} bound of the form 
\[
(1-\varepsilon) \mathbf{S} \preccurlyeq \widehat{\mathbf{S}} \preccurlyeq (1+\varepsilon) \mathbf{S}
\]
in high probability, where $\mathbf{S}$ is the unknown covariance matrix, and $\preccurlyeq$ is the positive semidefinite order on symmetric matrices. The result only requires the existence of fourth-order moments, and allows for $\varepsilon = O(\sqrt{\kappa^4 d\log(d/\delta)/n})$ where $\kappa^4$ is a measure of kurtosis of the distribution, $d$ is the dimensionality of the space, $n$ is the sample size, and $1-\delta$ is the desired confidence level. More generally, we can allow for regularization with level $\lambda$, then $d$ gets replaced with the degrees of freedom number. Denoting $\text{cond}(\mathbf{S})$ the condition number of $\mathbf{S}$, the computational cost of the novel estimator is $O(d^2 n + d^3\log(\text{cond}(\mathbf{S})))$, which is comparable to the cost of the sample covariance estimator in the statistically interesing regime $n \ge d$. We consider applications of our estimator to eigenvalue estimation with relative error, and to ridge regression with heavy-tailed random design.
\end{abstract}

%% file: intro.tex
\section{Introduction}
We are interested in estimating the covariance matrix~$\S = \E[X \otimes X]$~of a zero-mean random vector~$X \in \R^d$ from~$n$ independent and identically distributed~(i.i.d.)~copies~$X_1,..., X_n$ of~$X$. 
This task is crucial -- and often arises as a subroutine -- in some widely used statistical procedures, such as linear regression, principal component analysis, factor analysis, generalized methods of moments, and mean-variance portfolio selection, to name a few~\citep{friedman2001elements,jolliffe2002principal,hansen1982large,markowitz1952portfolio}.
In some of them, the control of~$\|\Shat - \S\|_*$, where~$\Shat$ is a covariance estimator and~$\|\cdot\|_*$ is the spectral, Frobenius or trace norm, does not result in sharp theoretical guarantees.
Instead, it might be necessary to estimate the eigenvalues of~$\S$ in relative scale, ensuring that
\[
|\la_j(\Shat) - \la_j(\S)| \leqs \veps \la_j(\S), \quad j \in \{1,...,d\},
\]
holds for $\veps > 0$ (this task arises in the analysis of the subspace iteration method, see~\cite{halko2011finding} and~\cref{sec:applications}).
More generally, one may seek to provide affine-invariant bounds of the form
\eqal{\label{eq:affine-invariant-intro}
(1-\epsilon) \S \preceq \Shat \preceq (1+\epsilon) \S,
}
as in the analysis of linear regression with random design~\citep[see][and \cref{sec:back-affine} for more details]{hsu2012random}, where $\preceq$ is the positive semidefinite partial order for symmetric matrices. 
In fact, the basic and very natural {\em sample covariance estimator} 
\[
\Stilde = \frac{1}{n} \sum_{i=1}^n X_i \otimes X_i
\]
can be shown to satisfy \cref{eq:affine-invariant-intro}
with probability at least $1-\delta$,~$\delta \in (0,1]$, and accuracy~$\veps$ scaling as~$O((\rank(\S)\log(d/\delta)/n)^{1/2})$, provided that~$X$ is subgaussian (see \cref{sec:back-affine} for a detailed discussion).
However, the assumption of sub-gaussianity might be too strong in the above applications.
Going beyond it and similar assumptions is particularly important in mathematical finance, where it is widely accepted that the prices of assets might have heavy-tailed distributions~\citep{kelly2014tail,bradley2003financial}. 

We propose a simple variation of the sample covariance estimator (see Algorithm~\ref{alg:calibrated-basic}) in the form
\[
\Shat = \frac{1}{n}\sum_{i=1}^n \alpha_i X_i \otimes X_i,
\]
where the coefficients $\alpha_1, \dots, \alpha_n > 0$ are chosen in a data-driven manner.
Our main result, stated informally below, shows that the proposed estimator enjoys \textit{high-probability} bounds analogous to those for the sample covariance estimator, under a weak moment assumption on the distribution.
%
Namely, we assume that for some~$\kappa \ge 1$ it holds
\eqal{
\label{ass:kurtosis}
\E^{1/4}[\scal{X}{u}^4] \leq \kappa \E^{1/2}[\scal{X}{u}^2], \quad \forall u \in \R^d.
\tag{HT}}
In other words, the {\em kurtosis} of~$X$ is bounded by~$\kappa$ in all directions.\footnotemark
\footnotetext{Note that we use a slightly non-standard definition of kurtosis, extracting the corresponding roots from the moments.}
Kurtosis is an affine-invariant and unitless quantity, and it is uniformly bounded from above by a constant for many common families of multivariate distributions: for example,~$\kappa = \sqrt[4]{3}$ for any Gaussian distribution, and~$\kappa \le \sqrt[4]{9}$ for the multivariate Student-t distribution with at least $5$ degrees of freedom.

Now we are ready to informally state our main result.
\bt[Simplified version of~\cref{th:calibrated-basic}]
Under~\eqref{ass:kurtosis}, there exists an estimator~$\Shat$ that has computational cost~$O(d^2n + d^3)$, and with probability at least $1-\delta$ satisfies~\eqref{eq:affine-invariant-intro} with accuracy
\eqal{\label{eq:affine-invariance}
\epsilon \leq 48\kappa^2\sqrt{\frac{\rank(\S)\log(4d/\delta)}{n}}.
}
\et
This result shows that the proposed estimator is a valid alternative to the sample covariance estimator: it has comparable accuracy and the same computational complexity, while requiring only boundedness of the fourth moment of~$X$ instead of sub-gaussianity.
More generally, by allowing a {\em regularization level}~$\la > 0$, i.e., using~$\Shat_{\lam} := \Shat + \la \Id$ instead of~$\Shat$ to estimate~$\S_{\la} := \S + \la \Id$ instead of~$\S$ (as required in ridge regression~\citep{hsu2012random}, we can replace~$\rank(\S)$ with the {\em degrees of freedom} number
\eqal{\label{def:degrees-of-freedom}
\df_\la(\S) := \tr (\S \S_{\lam}^{-1}).
}
This leads to a better bound, since~$\df_\la(\S)$ is never larger than~$\min\{\rank(\S), \tr(\S)/\la\}$, and can be way smaller depending on the eigenvalue decay of~$\S$: for example, if~$\la_j(\S) \leq j^{-b}$ with~$b \geq 1$, then~$\df_\la(\S) \leq \la^{-1/b}$.

\paragraph{Paper Organization.}
In~\cref{sec:background} we recall the known results for the sample covariance matrix estimator under light-tailed assumptions, together with some recent high-probability results for an alternative estimator applicable to heavy-tailed distributions. 
The novel estimator is presented in \cref{sec:algorithm} and analyzed and discussed in detail in \cref{sec:main-result}. 
In order to achieve the best statistical performance, it requires the knowledge of the distribution parameters~$\kappa$  and~$\df_{\lam}(\S)$ in advance;
in \cref{sec:lepski} we extend the algorithm, via a variant of Lepskii's method~\citep{lepski1991}, to be adaptive to these quantities.
Applications to eigenvalue estimation and ridge regression are discussed in~\cref{sec:applications}.

\paragraph{Notation and Conventions.}
For~$X \in \R^d$,~$X \otimes X$ denotes the outer product~$X X^\top$. W.l.o.g.~we assume that~$\S = \E[X \otimes X]$ is full-rank (otherwise we can work on its range).
To reduce the clutter of parentheses, we convene that powers and multiplication have priority over the expectation, and we denote the~$1/p$-th power of expectation,~$p \ge 1$, with~$\E^{1/p}[\cdot]$.
We use~$\|\cdot\|$ for the spectral norm of a matrix (unless specified otherwise), as well as for the~$\ell_2$-norm of a vector. We shortand~$\min(a,b)$ to~$a \wedge b$.
We use the~$O(\cdot)$ notation in a conventional way, and occasionally replace generic constants with~$O(1)$.
We use the notation~$\A_{\lam} := \A + \lam \Id$, where~$\A \in \R^{d \times d}$ and~$\Id$ is the identity matrix.

%% file: background.tex
\section{Background and Related Work}\label{sec:background}
In this section we recall some relevant previous work on covariance estimators for light-tailed and heavy-tailed distributions, and provide more intuition about affine-invariant error bounds. 
Moreover, we introduce basic concepts that will be used later on in the theoretical analysis.

\subsection{Relative Error Bounds for the Sample Covariance Estimator}
\label{sec:back-sample-cov}
When the estimation error is measured by~$\|\Stilde - \S\|$, where $\|\cdot\|$ is the \textit{spectral norm}, the problem can be reduced, via the Chernoff bounding technique, to the control of the matrix moment generating function, for which one can apply some deep operator-theoretic results such as the Goldon-Thompson inequality~\citep{ahlswede2002strong,oliveira2010sums} or Lieb's theorem~\citep{tropp2012user}. Alternatively, one may reduce the task to the control of an underlying empirical process, and exploit advanced tools from empirical process theory such as generic chaining~\citep{koltchinskii2017concentration}. 
Both families of approaches have been focused on the sample covariance estimator and its direct extensions~\citep{vershynin2010introduction,tropp2012user,tropp2015introduction}, requiring stronger assumptions on the distribution of~$X$ than~\eqref{ass:kurtosis}. 
In particular, consider the \textit{subgaussian moment growth} assumption
\begin{equation}
\label{ass:subgaussian}
\E^{1/p} [|\lang X, u \rang|^{p}] \le  \bar\kappa \sqrt{p} \, \E^{1/p} [\lang X, u \rang^{2}], \quad \forall u \in \R^d \;\; \text{and} \;\; p \ge 2,
\tag{SG}
\end{equation}
which implies~\eqref{ass:kurtosis} with~$\kappa = 2 \bar \kappa$, where~$\kappa$ is defined in \cref{ass:kurtosis}. 
Define the \textit{effective rank} of~$\S$ by
\begin{equation}
\label{def:eff-rank}
\reff(\S) := {\tr(\S)}/{\|\S\|}.
\end{equation}%
The following result is known.
\begin{theorem}[{Simplified version of~\cite[Prop.~3]{lounici2014high-dimensional}}]
\label{th:lounici}
Under~\eqref{ass:subgaussian}, the sample covariance estimator~$\wt \S$ with probability at least~$1-\delta$,~$\delta \in (0,1]$, satisfies
\begin{equation}
\label{eq:error-sample-covariance}
\| \wt \S - \S\| \le O(1) \bar\kappa^2 \|\S\| \sqrt{\frac{\reff(\S)\log(2d/\delta)}{n}},
\end{equation}
provided that~$n \ge \wt O(1) \reff(\S)$, where~$\wt O(1)$ hides polynomial dependency on~$\log(2d/\delta)$ and~$\log(2n/\delta)$.
\end{theorem}

It can be shown~\citep[Prop.~6.10]{ledoux2013probability} that~\cref{eq:error-sample-covariance} nearly optimally depends on~$\reff(\S)$,~$\bar \kappa$, and~$1/\delta$.\footnotemark~
\footnotetext{In fact,~\citep[Theorem 9]{koltchinskii2017concentration} replaces~$\reff(\S)\log(2d/\delta)$ by~$\reff(\S) + \log(1/\delta)$, making the bound dimension-independent, but does not specify the dependency on~$\bar \kappa$.
Similar results have been obtained under~\eqref{ass:kurtosis} for robust estimators (see~\cref{sec:background-robust}), which, however, are computationally intractable~\citep{mendelson2018robust}.}
Another remarkable property of this bound is that it is almost dimension-independent: 
up to a logarithmic factor, the complexity of estimating~$\S$ is independent of the ambient dimension~$d$. 
Instead, it is controlled by the distribution-dependent quantity~$\reff(\S)$, which always satisfies~$\reff(\S) \le \rank(\S) \le d$, and can be much smaller than~$\rank(\S)$ when the distribution of~$X$ lies close to a low-dimensional linear subspace, i.e., when~$\S$ has only a few relatively large eigenvalues. 

\subsection{Relative Error Bounds for Heavy-Tailed Distributions}
\label{sec:background-robust}
It is possible to obtain relative error bounds of the form $\|\Shat - \S\|$, including the ones in high probability, under \textit{weak} moment assumptions such as~\eqref{ass:kurtosis}, considering other estimators than the sample covariance matrix. 
In particular, ~\cite{wei2017estimation} propose an estimator based on the idea of clipping observations with large norm. 
Formally, they define the truncation map~$\psi_{\theta}: \R \to \R$,
\begin{equation}
\label{eq:truncation-map}
\psi_{\theta}(x) := (|x| \wedge \theta) \, \text{sign}(x),
\end{equation}
given a certain threshold~$\theta > 0$, and consider the estimator 
\begin{equation}
\label{def:minsker}
\wh\S^{\WM} := \frac{1}{n} \sum_{i=1}^n \rho_{\theta}(\|X_i\|) \, X_i \otimes X_i, \;\; \text{where} \;\; \rho_{\theta}(x) :=  {\psi_{\theta}(x^2)}/{x^2}.
\end{equation}
In other words, one simply truncates observations with squared norm larger than~$\theta$ prior to averaging. 
This estimator is a key ingredient in our Algorithm~\ref{alg:calibrated-basic}, and we now summarize its statistical properties. 
\begin{theorem}[{\citep[Lem.~2.1 and Lem. 5.7]{minsker2017estimation}}]
\label{th:minsker}
Define the matrix second moment statistic~$W := \|\E[\|X\|^2 X \otimes X]\|$.
Let~$\overline W \ge W$, and~$\delta \in (0,1]$. Then estimator~$\wh\S^{\WM}$, cf.~\eqref{def:minsker},
with~
$\theta = \sqrt{{n \overline W}/{\log(2d/\delta)}}$
with probability at least~$1-\delta$ satisfies~
$
\| \wh\S^{\WM} - \S \| \le 2\sqrt{{\overline W\log(2d/\delta)}/{n}}.
$
\end{theorem}

In contrast with~\cref{th:lounici},~\cref{th:minsker} claims \textit{subgaussian} concentration for the spectral-norm loss under the weak moment assumption~\eqref{ass:kurtosis}. 
Moreover, we arrive at the relative error bound akin to~\eqref{eq:error-sample-covariance}:
\begin{equation}
\label{eq:error-minsker}
\left\| \wh \S^{\WM} - \S \right\| \le 2 \kappa^2 \|\S\| \sqrt{\frac{\reff(\S) \log(2d/\delta)}{n}},
\end{equation}
if we bound the second moment statistic as
\begin{equation}
\label{eq:variance-bound}
W \le \left\| \E \|X\|^2 X \otimes X \right\| \le \k^4 \|\S\|^2 \reff(\S),
\end{equation}
see~\cite[Lem.~2.3 and Cor.~5.1]{wei2017estimation}, and choose the appropriate truncation level
\[
\theta = \kappa^2 \|\S\| \sqrt{{n \reff{(\S)}}/{\log(2d/\delta)}} \ge \sqrt{{nW}/{\log(2d/\delta)}}.
\]
Since this choice depends on the unknown~$W$, one can use a larger value, which will result in the inflation of the right-hand side of \cref{eq:error-minsker}. 
An alternative is to adapt to the unknown~$W$ via Lepskii's method~\citep{lepski1991} as described in~\cite[Thm~2.1]{wei2017estimation}. 
To conclude, the estimator~$\wh \S^{\WM}$ enjoys subgaussian relative error bounds under the fourth moment assumption~\eqref{ass:kurtosis}, while having essentially the same computation cost as the sample covariance estimator. 

\subsection{Affine-Invariant Bounds for the Sample Covariance Estimator}
\label{sec:back-affine}
As we have seen previously, estimator~$\wh\S^{\WM}$ has favorable statistical properties compared to~$\wt \S$ when the goal is to estimate~$\S$ in relative spectral-norm error as in~\cref{eq:error-minsker}. 
However, one can instead be interested in providing affine-invariant bounds in the form of \cref{eq:affine-invariant-intro}. 
More generally, one may wish to estimate~$\S$ only for the eigenvalues greater than some level~$\lam > 0$, that is, to guarantee that
\begin{equation}
\label{eq:intro-psd-regularized}
(1 - \veps)\S_{\lam} \prccq \wh \S_{\lam} \prccq (1+\veps)\S_{\lam}.
\end{equation}
The need for such bounds arises, in particular, in random-design ridge regression, where the information about inferior eigenvalues is irrelevant, since it is anyway erased by regularization. Note that~\cref{eq:intro-psd-regularized}, for any~$\lam > 0$, can be reformulated in terms of the~$\S_\lam^{-1/2}$-transformed spectral norm:
\begin{equation}
\label{eq:intro-relative-scale}
\left\|\S_\lam^{-1/2} (\wh\S - \S ) \S_\lam^{-1/2}\right\| \le \veps.
\end{equation}
The task of obtaining such bounds, with arbitrary regularization level~$0 \le \lam \le \|\S\|$, will be referred to as \textbf{calibrated covariance estimation}. Generally, this task is harder than proving relative-error bounds in the spectral norm such as~\cref{eq:error-minsker}: the latter is equivalent, up to a constant factor loss of accuracy, to proving~\cref{eq:intro-relative-scale} with~$\lam = \|\S\|$. 
On the other hand, calibrated covariance estimation also subsumes~\cref{eq:affine-invariant-intro} by taking~$\lam = O(\la_{\min}(\S))$, where~$\lmin(\S)$ is the smallest eigenvalue of~$\S$.

Now, one can make a simple observation that for the sample covariance estimator~$\wt\S$, calibrated bounds of the form~\eqref{eq:intro-relative-scale} ``automatically'' follow from the dimension-free spectral-norm bounds akin to~\eqref{eq:error-sample-covariance} or~\eqref{eq:error-minsker} due to its affine equivariance.
Indeed,~$\J := \S_\lam^{-1/2} \S \S_\lam^{-1/2}$
is precisely the covariance matrix of the ``$\lam$-decorrelated'' observations~$Z_i = \S_{\lam}^{-1/2}X_i$, for which the sample covariance estimator is given by~$\wt \J := \frac{1}{n} \sum_{i=1}^n Z_i \otimes Z_i = \S_{\lam}^{-1/2} \wt \S \S_{\lam}^{-1/2}$. Hence, we can apply the spectral-norm bound~\eqref{eq:error-sample-covariance}, replacing~$\S$ and~$\wt\S$ with~$\J$ and~$\wt \J$. Using the fact that~$\|\J\| \le 1$ for any~$\lam \ge 0$, and that assumptions~\eqref{ass:kurtosis},~\eqref{ass:subgaussian} are themselves invariant under (non-singular) linear transforms, we obtain
\begin{equation}
\label{eq:error-sample-covariance-relative}
\E^{1/2} \Big[\left\|\S_\lam^{-1/2} ( \wt\S - \S ) \S_\lam^{-1/2}\right\|^2\Big] \le O(1) \bar \k^2 \sqrt{\frac{\df_{\lam}(\S) \log(2d)}{n}}
\end{equation}
once~$n \ge \tO(1) \df_{\lam}(\S),$ where $\df_{\lambda}(\cdot)$ is defined in \cref{def:degrees-of-freedom} and ranges from~$O(\reff(\S))$ to~$\rank(\S) \le d$ as~$\lam$ decreases from~$\|\S\|$ to~zero. 
In fact, when~$\bar \k$ is a constant, this rate is known to be asymtptocially minimax-optimal over certain natural classes of covariance matrices, e.g., Toeplitz matrices with spectra discretizing those of H\"older-smooth functions~\citep{bickel2008regularized,cai2010optimal}. 
It is thus reasonable to ask whether one can extend~\cref{eq:error-sample-covariance-relative} in the same manner as~\cref{eq:error-minsker} extends~\cref{eq:error-sample-covariance}. 
In other words, \emph{can one provide a high-probability guarantee for the calibrated error~(cf.~\cref{eq:intro-relative-scale}) of the estimator~$\wh\S^{\WM}$ (cf.~\cref{def:minsker}) under assumption~\eqref{ass:kurtosis}?}
The immediate difficulty is that~$\wh\S^{\WM}$ -- in fact, the only estimator for which finite-sample high-probability guarantees under fourth moment assumptions are known to us -- does not allow for the same reasoning as~$\wt\S$ because of the non-linearity introduced by the truncation map.
On the other hand, the desired bounds are achieved by the ``oracle'' estimator that truncates the ``$\lam$-decorrelated'' vectors~$Z_i = \S_{\lam}^{-1/2}X_i$ with accordingly adjusted~$\theta$:
\begin{equation}
\label{def:oracle}
\wh\S^{o} := 
\frac{1}{n} \sum_{i=1}^n \rho_{\theta} (\|\S_{\lam}^{-1/2}X_i\|) X_i \otimes X_i 
= \S_{\lam}^{1/2} \left[\frac{1}{n} \sum_{i=1}^n \rho_{\theta}(\|Z_i\|) Z_i \otimes Z_i \right]\S_{\lam}^{1/2}. 
\end{equation}
cf.~\cref{def:minsker}. Unfortunately, this estimator is unavailable since~$Z_i$'s are not observable. 
In what follows, we present our main methodological contribution: an estimator that achieves the stated goal, and moreover, has a similar complexity of computation and storage as the sample covariance matrix.

\begin{remark}
Some robust covariance estimators, such as MCD or MVE~\citep{campbell1980robust,lopuhaa1991breakdown,rousseeuw1999fast}, are affine equivariant,
but to the best of our knowledge, the desired bounds are not known for them. On the other hand,~\cite{oliveira2016lower} shows that if one only seeks for the left-hand side bound in~\eqref{eq:intro-psd-regularized}, the sample covariance estimator suffices under~\eqref{ass:kurtosis}.
\end{remark}

%% file: algorithm.tex
\section{Proposed Estimator}
\label{sec:algorithm}
Our goal can be summarized as follows: given~$\delta \in (0,1],$ and~$\lam \ge 0$, provide an estimate~$\wh \S$ satisfying
\[
\left\|\S_\lam^{-1/2} ( \wh\S - \S ) \S_\lam^{-1/2}\right\| \le O(1)\k^2\sqrt{\frac{\df_{\lam}(\S) \log(2d/\delta)}{n}},
\]
where~$\kappa$ is the kurtosis parameter of~$X$ (cf.~\eqref{ass:kurtosis}). Moreover, we can restrict ourselves to the case~$\lam \le \|\S\|$, since otherwise the task is resolved by the estimator~$\wh \S^{\WM}$ as can be seen from~\cref{eq:error-minsker}.

As we have seen before, the oracle estimator~$\Shat^o$ introduced in the previous section (cf.~\cref{def:oracle}) achieves the stated goal, but is unavailable since it depends explicitly on~$\S$. 
The key idea of our construction is to approximate~$\Shat^o$ in an iterative fashion -- roughly, to start with~$\wh \S^{(0)} = \wh\S^{\WM}$, which is already a good estimate for~$\S_{\lam}$ with the crudest regularization level~$\lam = \|\S\|$ due to~\cref{eq:error-minsker}, and then iteratively refine the estimate by computing
\vspace{-0.1cm}
\begin{equation}
\label{eq:proto-update}
\wh \S^{(t+1)} = \frac{1}{n} \sum_{i=1}^n  \rho_{\theta} (\|[\Shat^{(t)}_{\lam}]^{-1/2}X_i\|) \, X_i \otimes X_i. 
\vspace{-0.2cm}
\end{equation}
To make this simple idea work, we need to adjust it in two ways.
Firstly,~$\wh\S^{(t)}$ depends on the observations~$X_i$, hence the random vectors~$[\wh\S^{(t)}_{\la}]^{-1/2} X_i$ are not independent.
To simplify the analysis, we split the sample into batches corresponding to different iterations, and at each iteration use observations of the new batch instead of~$X_i$'s in~\cref{eq:proto-update}. 
Since~$\wh \S^{(t)}$ is independent from the new observations, we can apply~\cref{th:minsker} conditionally at each step.

Secondly, as discussed before, the estimator~$\wh \S^{\WM}$ given by~\eqref{def:minsker} already solves the problem for~$\lam = \|\S\|$. To achieve~\cref{eq:intro-relative-scale} for a given $\la$, the idea is to start with~$\la_0 = \|\S\|$, and reduce~$\lam_t$ by a constant factor at each iteration, so that the error~$\|\S_{\lam}^{-1/2} ( \wh\S^{(t)} - \S ) \S_{\lam}^{-1/2}\|$ remains controlled for~$\lam = \lam_t$ at each step. 
This way we also ensure that the total number of iterations is logarithmic in~$\|\S\|/\lam$, and, in particular, logarithmic in the condition number~$\|\S\|/\lmin(\S)$ when~$\lam \ge \lmin(\S)$.

Algorithm~\ref{alg:calibrated-basic} presented below implements these ideas. Note that the final batch of observations takes a half of the overall sample: this is needed to achieve the best possible accuracy (up to a constant factor) for the final regularization level, while at the previous levels it suffices to maintain the accuracy~$\veps = 1/2$, and one can use smaller batches taking up a half of the sample in total, see Lem.~\ref{lem:stability} in~\cref{sec:main-result} for details. 
Once the coefficients~$\alpha_{i}^{(t)}$ at the given step have been computed, the new estimate~$\wh \S^{(t+1)}$ reduces to the sample covariance matrix of the weighted observations, which can be computed in time~$O(d^2m)$ where~$m$ is the size of the batch. The total cost of these computations in the course of the algorithm is thus~$O(d^2n)$. 
As for~$\alpha_{i}^{(t)}$, they are obtained by first performing the Cholesky decomposition~\citep{golub2012matrix} of~$\wh\S^{(t)}_{\la_t}$, i.e., finding the unique lower-triangular matrix~$\bR_t$ such that $\bR_t \bR_t^\top = \wh\S^{(t)}_{\la_t}$ which requires~$O(d^3)$ in time and~$O(d^2)$ in space,\footnotemark~and then computing each product~$\bR^{-1}_t X_i^{(t+1)}$ by solving the corresponding linear system in~$O(d^2)$.
\footnotetext{Choletsky decomposition is known to work whenever the condition number (in our case~$\|\S\|/{\lam}$) is dominated by the inverse machine precision; when this condition does not hold, one could add some extra tricks such as pivoting, which still results in~$O(d^3)$ complexity~\citep{golub2012matrix}.}
The total complexity of Algorithm~\ref{alg:calibrated-basic} is thus
\[
O\left(d^2 n + d^3 \log({L}/{\lambda})\right) \; \textrm{in time}, \quad O(d^2) \; \textrm{in space}.
\]
Moreover, the time complexity becomes $O(d^2 n)$ when~$n \gg d \log({L}/{\la})$; as we show next, this is anyway required to obtain a statistical performance guarantee.
Note moreover that it is possible to obtain the non-regularized version of \cref{eq:affine-invariant-intro} by choosing $\la = O(\la_{\min}(\S))$, 
in time~$O(d^2 n + d^3 \log(\textrm{cond}(\S)))$, where $\cond(\S) = \|\S\|/\lmin$ is the condition number of $\S$. 
\begin{remark}
In practice, sample splitting in Algorithm~\ref{alg:calibrated-basic} could be avoided, and iterations could be performed on the same sample.
We expect our statistical guarantees in~\cref{sec:main-result} to extend to this setup.
\end{remark}

Next we present a statistical guarantee for Algorithm~\ref{alg:calibrated-basic}, and suggest a way to select the parameters.


\begin{algorithm}
\caption{Robust Calibrated Covariance Estimation}
\begin{algorithmic}[1]
\label{alg:calibrated-basic}
\REQUIRE $X_1,..., X_n \in \R^d$,~$\delta \in (0,1]$, regularization level~$\lambda \le \|\S\|$,~$L \ge \|\S\|$, truncation level~$\theta > 0$
\STATE $\lam_0 = L,~T = \lceil \log_2(L/\lam) \rceil, ~m = \left\lfloor n/(2(T+1))\right\rfloor$
\STATE $\alpha^{(0)}_i = L \rho_\theta(\|X_i\|/\sqrt{L}),$ for~$i  \in \{1,\dots,m\}$, with~$\rho_\theta(x) = \psi_\theta(x^2)/x^2$ and~$\psi_{\theta}(\cdot)$ as in \cref{eq:truncation-map}
\STATE  $\Shat^{(0)} = \frac{1}{m} \sum_{i=1}^m \alpha^{(0)}_i X_i \otimes X_i$
\FOR{$t  = 0$ to $T-1$}
\STATE $(X^{(t+1)}_1:X^{(t+1)}_m) = (X_{m(t+1)+1}:X_{m(t+1)+m})$ \COMMENT{Obtain a new batch} 
\STATE $\bR_{t} = \textrm{Cholesky}(\Shat^{(t)}_{\la_{t}})$
\STATE $\alpha_i^{(t+1)} = \rho_\theta(\|\bR_{t}^{-1}X_i^{(t+1)}\|),$ for~$i  \in \{1,\dots,m\}$
\STATE $\Shat^{(t+1)} = \frac{1}{m} \sum_{i=1}^m \alpha_i^{(t+1)} X^{(t+1)}_i \otimes X^{(t+1)}_i$
\STATE $\lam_{t+1} = {\lam_{t}}/{2}$
\ENDFOR
\STATE $r = n-m(T+1)$ \COMMENT{Size of the remaining sample, roughly~$n/2$}
\STATE $\theta_T = 2\theta (T+1)^{1/2}\left(1+\log(T+1)/\log(4d/\delta)\right)^{1/2}$ \COMMENT{Final truncation level}
\label{alg-line:final-truncation}
\STATE $(X^{\star}_1:X^{\star}_r) = (X_{mT+1}:X_{n})$ \COMMENT{Remaining sample} 
\STATE $\bR_T = \textrm{Cholesky}(\Shat^{(T)}_{\la_{T}})$
\STATE $\alpha^\star_i = \rho_{\theta_T}(\|\bR_T^{-1} X^\star_{i}\|),$ for~$i  \in \{1,\dots,r\}$
\STATE $\Shat^\star = \frac{1}{r} \sum_{i=1}^r \alpha^\star_i X^\star_i \otimes X^\star_i$ \COMMENT{Final estimate}
\ENSURE $\Shat^\star$
\end{algorithmic}
\end{algorithm}


%% file: theory.tex
\section{Statistical Guarantee}
\label{sec:main-result}

In Theorem~\ref{th:calibrated-basic} below, we show that the estimator produced by Algorithm~\ref{alg:calibrated-basic} achieves a high-probability bound of the type~\eqref{eq:intro-relative-scale}
requiring only the existence of the fourth-order moments of~$X$, and the correct choice of the truncation level~$\theta$.
We begin with the lemma that justifies the proposed update rule.

\begin{lemma}
\label{lem:stability}
Let $\wh\S$ be a symmetric estimate of $\S$ such that, for some~$\lambda > 0$,
\begin{equation}
\label{eq:stability-premise}
\|\S_{\lam}^{-1/2}(\wh\S-\S)\S_{\lam}^{-1/2}\| \le 1/2.
\end{equation}
Conditioned on $\Shat$, let~$X_1,\dots,X_m$ be i.i.d., have zero mean, covariance~$\S$, and finite fourth-order moments. Let $\kappa$ be the associated (conditional) kurtosis as in \cref{ass:kurtosis}. Define $\S^{(+)}$ as
\begin{equation}
\label{eq:stability-step}
\wh \S^{(+)} := \frac{1}{m} \sum_{j=1}^m \rho_{\theta}(\|\wh\S_{\lam}^{-1/2} X_j\|) X_j \otimes X_j, 
\end{equation}
with $\rho_\theta$ defined in Eqs.~\eqref{eq:truncation-map}--\eqref{def:minsker}. 
Choose~$\theta \ge 2\sqrt{2}\k^2\sqrt{{m \df_{\lam}(\S)}/{\log(2d/\delta)}},$
where~$\df_{\lam}(\S)$ is defined by~\cref{def:degrees-of-freedom}. 
Then with conditional probability at least~$1-\delta$ over~$(X_1,...,X_m)$ it holds
\begin{equation}
\label{eq:stability-result}
\left\|\S_{\lam/2}^{-1/2}(\wh\S^{(+)}-\S )\S_{\lam/2}^{-1/2}\right\| \le \frac{6\theta \log(2d/\delta)}{m}.
\end{equation}

\end{lemma}

Lemma~\ref{lem:stability} is proved in Appendix~\ref{sec:proof-stability}. 
Its role is to guarantee the stability of the iterative process in Algorithm~\ref{alg:calibrated-basic} when we pass to the next regularization level by~$\lambda^{(t)} \gets \lambda^{(t-1)}/2$. Indeed, if the size~$m$ of the new batch is large enough, the right-hand side of~\cref{eq:stability-result} can be made smaller than~$1/2$, which allows to apply Lemma~\ref{lem:stability} sequentially. 
We are now ready to present the guarantee for Algorithm~\ref{alg:calibrated-basic}.
\begin{theorem}
\label{th:calibrated-basic}
Let $X_1, \dots, X_n$ be i.i.d. zero-mean~random vectors in $\R^d$ satisfying~$\E[X_i \otimes X_i] = \S$ and the kurtosis assumption~\eqref{ass:kurtosis}. 
Let Algorithm~\ref{alg:calibrated-basic} be run with~$\delta \in (0,1]$,~$0 < \lam \le \|\S\| \le L$, and
\begin{equation}
\label{eq:calibrated-basic-threshold}
\theta \ge \theta_* := 2\k^2\sqrt{\frac{n \df_{\lam}(\S)}{\q \log(4\q d/\delta)}}, \quad \text{where} \quad \q := \lceil \log_2(L/\lam) \rceil + 1.
\end{equation}
Whenever the sample size satisfies
\begin{equation}
\label{eq:calibrated-basic-min-sample-size}
n \ge 48 \q \theta \log(4\q d/\delta),
\end{equation}
the resulting estimator~$\wh \S$ with probability at least~$1-\delta$ satisfies
\begin{equation}
\label{eq:calibrated-basic-result}
\left\|\S_\lam^{-1/2} ( \wh\S - \S ) \S_\lam^{-1/2}\right\| \le \frac{24\theta \sqrt{\q\log(4\q d/\delta)\log(4d/\delta)}}{n}.
\end{equation}
\end{theorem}

The above theorem shows that when the conditions on $\theta$ and $n$ are met, the proposed estimator satisfies an affine-invariant error bound with accuracy of the same order as the one available for the sample covariance estimator (cf.~\cref{eq:error-sample-covariance-relative}) under the more stringent sub-gaussian assumption.
This is made explicit in the next corollary, where we simply put $\theta = \theta_*$ as suggested by~\cref{eq:calibrated-basic-threshold}, 
obtaining the bound (cf.~\cref{eq:calibrated-basic-result-simplified}) that matches \cref{eq:error-sample-covariance-relative} up to a constant factor and the replacement of~$\bar \kappa$ with~$\kappa$.

\begin{corollary}
Under the premise of Theorem~\ref{th:calibrated-basic}, assume that
\begin{equation}
\label{eq:calibrated-basic-min-sample-size-simplified}
n \ge 96^2 \k^4 \q \df_{\lam}(\S) \log(4\q d/\delta),
\end{equation}
and choose~$\theta = \theta_*$, cf.~\cref{eq:calibrated-basic-threshold}. Then the estimator given by Algorithm~\ref{alg:calibrated-basic}  w.p. at least~$1-\delta$ satisfies
\begin{equation}
\label{eq:calibrated-basic-result-simplified}
\left\|\S_\lam^{-1/2} ( \wh\S - \S ) \S_\lam^{-1/2}\right\| \le 48\k^2\sqrt{\frac{\df_{\lam}(\S)\log(4d/\delta)}{n}}.
\end{equation}
\end{corollary}
We conclude with a remark on choosing~$L$, while in \cref{sec:lepski} we will provide an adaptive version of the estimator based on a version of Lepskii's method~\cite{lepski1991} in which~$\theta$ is tuned automatically.
\begin{remark}[Choosing~$L$]
One can simply put~$L = 2\|\wh\S^{\WM}\|$, requiring that $\la \leq \frac{2}{3}\|\wh\S^{\WM}\|$, with $\wh\S^{\WM}$ defined in~\cref{def:minsker} and using an independent subsample of size~$n/(T+1)$. Indeed, under \cref{eq:calibrated-basic-min-sample-size}, the result of Theorem~\ref{th:minsker} ensures that~$\frac{2}{3}\|\wh\S^{\WM}\| \le \|\S\| \le 2\|\wh\S^{\WM}\|$ with probability~$\ge 1-\delta$.
\end{remark}

\subsection{Proof of Theorem~\ref{th:calibrated-basic}}
Note that~$\q = T + 1$ is the number of batches processed by the end of the for-loop in~Algorithm~\ref{alg:calibrated-basic}. Thus, using that~$m = \lfloor n/ (2\q) \rfloor$ and~$n \ge 4\q$ (see~\cref{eq:calibrated-basic-threshold}--\eqref{eq:calibrated-basic-min-sample-size} and use that~$\kappa \ge 1$,~$\df_{\lam} \ge 1$), we get
\begin{equation}
\label{eq:subsample-size-bound}
{n}/{(4\q)} \le m \le {n}/{(2\q)}.
\end{equation}
We will proceed by induction over the steps~$0 \le t \le T$, showing that
\begin{equation}
\label{eq:calibrated-basic-induction-goal}
\left\|\S_{\lam_t}^{-1/2} ( \wh\S^{(t)} - \S ) \S_{\lam_t}^{-1/2} \right\| \le {1}/{2} 
\end{equation}
holds for~$t = 0, ..., T$ with probability~$\ge (1-\delta/(2\q))^{t+1}$. Then we will derive~\cref{eq:calibrated-basic-result} as a corollary.

\proofpoint{1}
For the base, we can apply~\cref{th:minsker}, exploiting that~$\wh\S^{(0)} = L\wh\S^{\WM}$ for the (renormalized) initial batch~$\frac{1}{\sqrt{L}}(X_1, ..., X_m)$. 
Thus, with probability at least~$1-{\delta}/{(2\q)}$ over this batch, it holds
\begin{equation}
\label{eq:calibrated-basic-induction-base}
\frac{1}{L} \big\|\wh\S^{(0)} - \S \big\| \le \frac{2\theta \log(4\q d/\delta)}{m} \le \frac{8\q\theta\log(4\q d/\delta)}{n},
\end{equation}
provided that (recall the condition~
$\theta \ge \sqrt{{n W}/{\log(2d/\delta)}}$ in~\cref{th:minsker}
and combine it with~\cref{eq:variance-bound}):
\[
\theta \ge \k^2 {\|\S\|}/{L} \cdot \sqrt{{m\reff(\S)}/{\log(4\q d/\delta)}}.
\]
But this follows from~\cref{eq:calibrated-basic-threshold}, since~$\|\S\| \le L$,~$m \le n/(2\q)$, and~$\reff(\S) \le 2\df_{\|\S\|}(\S) \le 2\df_{\lam}(\S).$ 
Noting that~$\|\S_{L}^{-1}\| \le 1/L$, from~\cref{eq:calibrated-basic-induction-base,eq:calibrated-basic-min-sample-size} we get
\[
\left\|\S_{L}^{-1/2} (\wh\S^{(0)} - \S ) \S_{L}^{-1/2} \right\| \le \frac{8\q\theta\log(4\q d/\delta)}{n} \le \frac{1}{6}.
\]
Since~$\lam_0 = L$, the induction base is proved. Note that when~$T = 0$, this already results in~\cref{eq:calibrated-basic-induction-goal}.

\proofpoint{2}
Let~$T \ge 1$ and~$0 \le t \le T-1$. For the induction step, we apply Lemma~\ref{lem:stability} conditionally on the first~$t$ iterations, with~$\wh\S^{(t)}$ in the role of the current estimate,~$\lam_{t}$ in the role of the current regularization level, and~$(X_{1}^{(t+1)},..., X_m^{(t+1)})$ as the new batch (which is independent from~$\wh\S^{(t)}$ by construction). 
By the induction hypothesis, we have~\cref{eq:calibrated-basic-induction-goal} with conditional probability~$\ge (1-\delta/(2\q))^{t+1}$. 
By Lemma~\ref{lem:stability}, since~$\lam_t \ge \lam$ (and thus~$\df_{\lam_t}(\S) \le \df_{\lam}(\S)$),~\cref{eq:calibrated-basic-threshold}, when combined with the upper bound~in~\cref{eq:subsample-size-bound}, guarantees that with conditional probability~$\ge 1-\delta/(2\q)$ over the new batch, 
\[
\left\|\S_{\lam_{t+1}}^{-1/2} (\wh\S^{(t+1)}-\S )\S_{\lam_{t+1}}^{-1/2}\right\| \le {6\theta \log(4\q d/\delta)}/{m} \le {24\q \theta \log(4\q d/\delta)}/{n} \le {1}/{2}.
\]
Here in the second transition we used the lower bound of~\cref{eq:subsample-size-bound}, and in the last transition we used~\cref{eq:calibrated-basic-min-sample-size}. 
Thus, the induction claim is proved. 
In particular, we have obtained that the bound
\[
\left\|\S_{\lam_T}^{-1/2} ( \wh\S^{(T)} - \S ) \S_{\lam_T}^{-1/2}\right\| \le {1}/{2}
\]
holds with probability at least~$(1-\delta/(2\q))^{\q} \ge 1-\delta/2$ over the first~$\q = T+1$ batches. 

\proofpoint{3}
Finally, to obtain \cref{eq:calibrated-basic-result}, we apply Lemma~\ref{lem:stability} once again, this time conditioning on~$\wh\S^{(T)}$, and using the last batch~$(X_{n-r+1},..., X_{n})$ with the final estimator~$\wh \S$ in the role of~$\wh\S^{(+)}$.
Note that the first condition~\cref{eq:stability-premise} of Lemma~\ref{lem:stability} follows from the just proved induction claim. 
On the other hand, by~\cref{eq:calibrated-basic-threshold} the final truncation level~$\theta_T$, cf.~line~\ref{alg-line:final-truncation} of Algorithm~\ref{alg:calibrated-basic}, satisfies
\[
\theta_T = 2\theta\sqrt{{\q\log(4\q d/\delta)}/{\log(4d/\delta)}} \ge 4\k^2 \sqrt{{n \df_{\lam}(\S)}/{\log(4d/\delta)}}. 
\]
The number of degrees of freedom is a stable quantity: we can easily prove (see Lem.~\ref{lem:df} in Appendix) that~$\df_{\lam/2}(\S) \le 2\df_{\lam}(\S)$. 
On the other hand,~$\df_{\lam_{T}} \le \df_{\lam/2}$ since~$\lam_T \ge \lam/2$. 
Using that, we have
\[
\theta_T \ge 2\sqrt{2}\k^2 \sqrt{{n \df_{\lam_{T}}(\S)}/{\log(4d/\delta)}} \ge 2\sqrt{2}\k^2 \sqrt{{r \df_{\lam_{T}}(\S)}/{\log(4d/\delta)}},
\]
meeting the requirement on the truncation level imposed in Lemma~\ref{lem:stability}.
Applying the lemma, and using that~$r \ge n/2$, we get that with conditional probability~$\ge 1-\delta/2$ over the last batch~$(X_{n-r+1},..., X_{n})$,
\[
\left\|\S_{\lam_{T}/2}^{-1/2} ( \wh\S - \S ) \S_{\lam_T/2}^{-1/2}\right\| 
\le \frac{6\theta_T \log(4d/\delta)}{r} \le \frac{24\theta \sqrt{\q\log(4\q d/\delta) \log(4d/\delta)}}{n}.
\]
Since~$\lam_T \le \lam$ implies~$\|\S_{\lam_T/2} \S_{\lam}^{-1}\| \le 1$, by the union bound we arrive at~\cref{eq:calibrated-basic-result}. 
\hfill \qed

\section{Adaptive Estimator}
\label{sec:lepski}

One limitation of Algorithm~\ref{alg:calibrated-basic} is that the truncation level~$\theta$ has to be chosen in advance in order to obtain the optimal statistical performance (see Theorem~\ref{th:calibrated-basic}), and the optimal choice~$\theta_*$ (see~\cref{eq:calibrated-basic-threshold}) depends on~$\kappa,~\df_{\lam}(\S)$ that are usually unknown.
To address this, we propose an \textit{adaptive estimator} (see Algorithm~\ref{alg:calibrated-adaptive}), in which~Algorithm~\ref{alg:calibrated-basic} is combined with a Lepskii-type procedure~\citep{lepski1991}, resulting in a near-optimal guarantee without the knowledge of~$\theta_*$.
Namely, let us be given a range~$0 < \theta_{\min} \le \theta_{\max}$ known to contain~$\theta_*$ but possibly very loose, and define the logarithmic grid
\begin{equation}
\label{eq:lepski-grid}
\theta_j =  2^j \theta_{\min}, \;\; \text{where} \;\; j  \in  \cJ := \{j \in \Z: \thetamin \le \theta_j  \le 2\thetamax\}.
\end{equation}
Define also
\begin{equation}
\label{eq:lepski-bounds}
\veps_j := \frac{24\theta_j \sqrt{\q \cdot \log(4\q d |\cJ| /\delta) \cdot \log(4d |\cJ| /\delta)}}{n}.
\end{equation}
Later on we will we show (see \cref{th:calibrated-basic}) that~$\veps_j$ is the error bound, with probability at least~$1-\delta$, for the estimator produced by Algorithm~\ref{alg:calibrated-basic} with truncation level~$\theta = \theta_j$. 
In Algorithm~\ref{alg:calibrated-adaptive}, we first compute~$\veps_j$ and the basic estimators~$\wh \S_{j} := \wh \S[\theta_j]$ for all truncation levels~$\theta_j$, then select
\begin{equation}
\label{eq:lepski-choice}
\whj = \min \left\{j \in \cJ: \;\; \forall j' \ge j \;\; \text{s.t.} \;\; j' \in \cJ \;\; \text{it holds} \;\; \left\| \wh\S_{j',\la}^{-1/2} (\wh\S_{j'} - \wh\S_{j}) \wh\S_{j',\la}^{-1/2} \right\| \le 2(\veps_{j'} + \veps_j) \right\},
\end{equation}
and output~$\wh \S_{\whj}$ as the final estimator. 
In~\cref{th:calibrated-adaptive} below, we show that this estimator admits essentially the same statistical guarantee (in the sense of~\cref{eq:intro-relative-scale}) as the ``ideal'' estimator which uses~$\theta = \theta_*$.


\begin{algorithm}
\caption{Robust Calibrated Covariance Estimation with Adaptive Truncation Level}
\begin{algorithmic}[1]
\label{alg:calibrated-adaptive}
\REQUIRE{$X_1,..., X_n \in \R^d$,~$\delta \in (0,1]$, regularization level~$\lambda \le \|\S\|$,~$L \ge \|\S\|$, range~$[\theta_{\min}, \theta_{\max}]$}
\STATE Form the grid~$\cJ := \{j \in \Z: \thetamin \le \theta_j  \le 2\thetamax\}$
\FOR{$j \in \cJ$}
\STATE Compute the output~$\wh \S_{j}$ of Algorithm~\ref{alg:calibrated-basic} with truncation level~$\theta_j = 2^j \thetamin$; compute~$\veps_j$ by~\eqref{eq:lepski-bounds} 
\ENDFOR
\ENSURE{$\Shat_{\whj}$ with~$\whj \in \cJ$ selected according to~\eqref{eq:lepski-choice}}
\end{algorithmic}
\vspace{-0.2cm}
\end{algorithm}


Next we present a statistical performance guarantee for Algorithm~\ref{alg:calibrated-adaptive}. 
Its proof, given in Appendix~\ref{sec:proof-lepski}, hinges upon the observation that the matrix~$\wh\S_{j',\la}$ in the error bound of~\cref{eq:lepski-choice} can essentially be replaced with its unobservable counterpart~$\S_{\la}$; this makes the analyzed errors \textit{additive}, so that the usual argument for Lepskii's method could be applied.

\begin{theorem}
\label{th:calibrated-adaptive}
Assume~\eqref{ass:kurtosis}, and let Algorithm~\ref{alg:calibrated-adaptive} be initialized with~$\lam \le \|\S\|$,~$L \ge \|\S\|$,~$\delta \in (0,1]$, and a range~$[\thetamin, \thetamax]$ containing the optimal truncation level~$\theta_*$ given by~\cref{eq:calibrated-basic-threshold}.
Moreover, let 
\begin{equation}
\label{eq:calibrated-lepski-min-sample-size-thetamax}
n \ge 96 \q \thetamax \log(4 \q d|\cJ|/\delta),
\end{equation}
where~$\q$ is defined in~\cref{eq:calibrated-basic-threshold}, and~$|\cJ| \le 1 + \log_2(\thetamax/\thetamin)$ is the cardinality of the grid defined in~\cref{eq:lepski-grid}.
Then the estimator~$\wh \S_{\whj}$ produced by Algorithm~\ref{alg:calibrated-adaptive} with probability at least~$1-\delta$ satisfies
\begin{equation*}
\left\|\S_\lam^{-1/2} ( \wh \S_{\whj} - \S ) \S_\lam^{-1/2}\right\| \le 720\k^2\sqrt{\frac{(1+\rho)\df_{\lam}(\S)\log (4d |\cJ| /\delta)}{n}}, \quad \text{where} \;\;  \rho :=  \frac{\log |\cJ|}{\log(4 \q d/\delta)}.
\end{equation*}
\end{theorem}

From the result of~\cref{th:calibrated-adaptive}, we see that the adaptive estimator nearly attains the best possible stastistical guarantee that corresponds to the optimal value of the truncation level, up to the iterated logarithm of the ratio~$\thetamax/\thetamin$. 
However, the premise of~\cref{th:calibrated-adaptive} requires~$\thetamax$ to be bounded both from above and below (cf.~\cref{eq:calibrated-basic-threshold,eq:calibrated-lepski-min-sample-size-thetamax}), and the two bounds are compatible only starting from a certain sample size.
Next we state a corollary of~\cref{th:calibrated-adaptive} that explicitly specifies the required sample size (the requirement is similar to~\cref{eq:calibrated-basic-min-sample-size-simplified}), and provides a reasonables  choice of~$[\thetamin, \thetamax]$. 

\begin{corollary}
\label{cor:lepski-simple}
Assume that we have~
\begin{equation}
\label{eq:lepski-simple-condition}
n \ge 192^2 (1+\rho_J) \kappa^4 \q \df_{\lam}(\S) \log(4\q dJ/\delta),
\end{equation}
where~
$J := 1 + \left\lceil \frac{1}{2}\log_2\left({n}/{96 \q} \right) \right\rceil$ and~$\rho_J := {\log(J)}/{\log(4\q d/\delta)}.$
Then the premise of Theorem~\ref{th:calibrated-adaptive} holds for the grid~$\cJ$ with cardinality~$J$ defined by
$\thetamax = {n}/{96\q\log(4\q d J/\delta)}$,~$\thetamin = 2^{1-J} \thetamax.$
\end{corollary}

%% file: applications.tex
\section{Applications}
\label{sec:applications}

\subsection{Relative-Scale Bounds for Eigenvalues}
\label{sec:eigenvalues}

Recall that the bounds obtained in~\cref{th:calibrated-basic,th:calibrated-adaptive} for the estimators~$\wh \S$ given by Algorithms~\ref{alg:calibrated-basic}--\ref{alg:calibrated-adaptive} read
\begin{equation}
\label{eq:app-psd-regularized}
(1 - \veps)\S_{\lam} \prccq \wh \S_{\lam} \prccq (1+\veps)\S_{\lam}
\end{equation}
for certain accuracy~$\veps < 1/2$ and regularization level~$\lam \ge 0$, provided that the sample size is large enough.
Using that the positive-semidefinite order preserves the order of eigenvalues, we obtain the corollary of~\cref{th:calibrated-basic} for eigenvalue estimation (\cref{th:calibrated-adaptive} has a similar corollary, which we omit).
\vspace{-0.2cm}
\begin{corollary}
\label{cor:evals-estimation}
Assume that~$n$ satisfies~\cref{eq:calibrated-basic-min-sample-size-simplified},
and let~$\wh \S$ be given by Algorithm~\ref{alg:calibrated-basic} with the optimal choice of the truncation level~$\theta = \theta_*$, cf.~\cref{eq:calibrated-basic-threshold}. Let also~
$
\|\S\| = \lam_1 \ge ... \ge \lam_d = \lmin
$ 
be the ordered eigenvalues of~$\S$, and~$\wh \lam_1 \ge ... \ge \wh \lam_d$ those of~$\wh \S$. 
Finally, assume that the regularization level in Algorithm~\ref{alg:calibrated-basic} satisfies~$\lam \le \lam_{k}$ for some~$1 \le k \le d$.
Then, with probability at least~$1-\delta$ it holds
\begin{equation}
\label{eq:eval-bound}
(1 - 2\veps) \lam_i \le \wh \lam_i \le (1+2\veps) \lam_i, \quad \text{for any} \;\; 1 \le i \le k,
\end{equation}
with~$\veps$ given by~\eqref{eq:calibrated-basic-result-simplified}. As a consequence, we have~
$
{(1-2\veps)^2} \cdot {\lam_i}/{\lam_k} \le {\wh \lam_i}/{\wh \lam_{k}} \le {(1-2\veps)^{-2}} \cdot {\lam_i}/{\lam_k}. 
$

\end{corollary}
As a simple application of this result, consider the task of ``noisy'' principal component analysis (PCA), i.e., performing PCA for the unknown covariance matrix~$\S$ from the observations~$X_1, ..., X_n$. 
A common way to approach it is by performing \textit{subspace iteration}~\citep{halko2011finding,mitliagkas2013memory,hardt2014noisy,balcan2016improved} with the estimated covariance~$\wh \S$: randomly choose~$U^{(0)} \in \R^{d \times k}$, and then iteratively multiply~$U^{(t)}$ by~$\wh \S$ and orthonormalize the result until convergence.
The iterate converges to the projector on the subspace of the top~$k$ eigenvalues of~$\wh\S$ (providing an estimate of the corresponding subspace for~$\S$), and its rate of convergence is known to be controlled by the ratio~$\wh\lam_{k}/\wh\lam_{k+1}$.\footnotemark~
Hence, if we use the estimate~$\wh \S$ produced by Algorithm~\ref{alg:calibrated-basic} or Algorithm~\ref{alg:calibrated-adaptive}, and if~$n$ is sufficient to guarantee that~$\veps < 1/2$,
the convergence rate to the top-$k$ eigenspace of~$\wh \S$ will essentially be the same as that for the exact method and the target subspace of~$\S$.
\footnotetext{One can use~$U^{(t)} \in \R^{d \times r}$,~$r \ge k$; the convergence rate is then controlled by the ratio of \textit{non-sequential} eigenvalues.} 

\subsection{Ridge Regression with Heavy-Tailed Observations}
\label{sec:least-squares}

In random design linear regression~\citep{hsu2012random}, one wants to fit the linear model~$Y = X^\top w$ from i.i.d.~observations~$(X_i, Y_i) \in \R^d \times \R$,~$1 \le i \le n$. 
More previsely, the goal is to find a minimizer~$w^* \in \R^d$ of the quadratic risk~$L(w) := \E (Y - X^\top w)^2$,
where the expectation is over the test pair~$(X, Y)$ independent of the sample and coming from the same distribution. 
In ordinary ridge regression, one fixes the regularization level~$\lam \ge 0$, and estimates~$w^*$ with the minimizer of the regularized empirical risk~$\wt L_{\lam}(w) :=  \frac{1}{n} \sum_{i=1}^n (Y_i^{\vphantom \top} - X_i^\top w)^2 + \lam \|w\|^2$.
The case~$\lam = 0$ corresponds to the ordinary least-squares estimator, while~$\lam > 0$ allows for some bias.

Here we propose a couterpart of this estimator with a favorable statistical guarantee under fourth-moment assumptions on the design and response. 
Given the sample~$X_1, ..., X_{2n}$, we first compute the covariance estimator~$\wh \S$ by feeding the hold-out sample~$X_{n+1}, ..., X_{2n}$ to Algorithm~\ref{alg:calibrated-basic} with~$\theta = \theta_*$ (one could also use Algorithm~\ref{alg:calibrated-adaptive}). 
Then, using the first half of the observations, we construct the ``pseudo-decorrelated'' observations~$(\wh Z_1,..., \wh Z_n)$ with~$\wh Z_i = \wh\S_{\lam}^{-1/2} X_i Y_i$, and compute the estimator
\begin{equation}
\label{eq:ridge-estimator-robust}
\bar w_{\lam} = \wh \S_\lam^{-1/2} \bar Z, \quad \text{where} \quad \bar Z = \frac{1}{n} \sum_{i=1}^n \rho_{\bar\theta}(\|\wh Z_i\|^{1/2}) \wh Z_i.
\end{equation}
Here,~$\rho_{\bar{\theta}}(\cdot)$ is as in~\cref{eq:truncation-map}--\eqref{def:minsker}, and~$\bar \theta$ is defined later.
We prove the following result (see  Appendix~\ref{sec:proof-ridge}).

\begin{theorem}
\label{th:ridge-robust}
In the above setting, assume that $X$ satisfies~$\E[X] = 0$,~$\E[X \otimes X] = \S$, and assumption~\eqref{ass:kurtosis}, and that~$Y$ has finite second and fourth moments:~$\E[Y^2] \le v^2$,~$\E[Y^{4}] \le \kresp^4 v^4$.
Assume also that~$n$ satisfies~\cref{eq:calibrated-basic-min-sample-size-simplified} from the premise of~\cref{th:calibrated-basic}. 
Then, the estimator~$\bar w_{\lam}$ given by~\eqref{eq:ridge-estimator-robust} with~
$
\bar \theta = \sqrt{{n \kappa^2 \kresp^2 v^2 \df_{\lam}(\S)}/{\log(1/\delta)}}
$
with probability at least~$1-\delta$ satisfies
\begin{equation}
\label{eq:ridge-result}
L(\bar w_{\lam}) - L(w^*) \le O(1) \left[ (\kappa^4 + \kappa^2 \kresp^2)  \frac{v^2 \df_{\lam}(\S) \log(2d/\delta)}{n} + \lam^2 \left\| \S_{\lam}^{-1/2}  w^*\right\|^2 \right].
\end{equation}
\end{theorem}

In the above result, the bias term is correct (leading to the minimax-optimal rates in the fixed design setting), and the stochastic term has the asymptotically optimal scaling~$O(\deff\log(1/\delta)/n)$, see, e.g., \cite{caponnetto2007optimal}. However, the obtained bound depends on the second moment~$v^2$ of the response instead of its variance. 
We believe that this problem could be resolved, leading to the fully optimal result, by replacing the truncated estimator~$\bar Z$ in~\cref{eq:ridge-estimator-robust} with median-of-means.

%% file: conclusion.tex
\section{Conclusion}
In this work, we have provided an estimator of the covariance matrix of a heavy-tailed multivariate distribution that admits a high-probability bound of the type
\[
(1-\varepsilon) \mathbf{S} \preccurlyeq \widehat{\mathbf{S}} \preccurlyeq (1+\varepsilon) \mathbf{S}.
\]
The novel estimator is computationally efficient, and has applications in principal component analysis, and in ridge regression with heavy-tailed random design. 
Let us now point out possible directions for future work.

First, one could investigate even weaker moment assumptions than~\eqref{ass:kurtosis}, for example, assuming the existence of the~$2+\eps$ moment of the one-dimensional marginals of~$X$. We envision no principal obstacles in extending our work in this direction.

Second, it would be interesting to reach full optimality in ridge regression with heavy-tailed design, replacing the second moment~$v^2$ of the response with its variance~$\sigma^2$ (see \cref{eq:ridge-result}). 
To the best of our understanding, a somewhat different expansion of the excess risk would be needed to achieve this; 
however, we believe that all tools needed to prove such a result are already in place, and it only remains to combine them in a right way.

%% file: acks.tex
\section*{Acknowledgments}
The first author has been supported by the ERCIM Alain Bensoussan Fellowship and the ERC grant SEQUOIA 724063.
The second author acknowledges support from the ERC grant SEQUOIA 724063.
We thank Zaid Harchaoui, Anatoli Juditsky, and Francis Bach for fruitful discussions. Finally, we thank Nikita Zhivotovskiy and anonymous reviewers for their insightful remarks, as well as for pointing out some relevant literature.

%% file: df-lemma.tex
\section{Degrees of Freedom Lemma}

\begin{lemma}
\label{lem:df}
For any~$\S \succq 0$ and~$\lam \ge 0$, define~$\df_{\lam}(\S) := \tr(\S \S_{\lam}^{-1})$. Then,~$\df_{\lam/2}(\S) \le 2\df_{\lam}(\S)$. 
\end{lemma}
\begin{proof}
We have
\[
\begin{aligned}
|\df_{\lam/2}(\S) - \df_{\lam}(\S)| 
&= \left|\tr\left[\S\left(\S_{\lam}^{-1} - \S_{\lam/2}^{-1}\right)\right]\right| \\
&= \left|\tr\left[\S_{\lam}^{-1/2}\S\S_{\lam}^{-1/2}\left(\Id - \S_{\lam}^{1/2}\S_{\lam/2}^{-1}\S_{\lam}^{1/2}\right)\right]\right| \\
&\le \df_{\lam}(\S) \left\|\Id - \S_{\lam}^{1/2}\S_{\lam/2}^{-1}\S_{\lam}^{1/2}\right\| \le \df_{\lam}(\S),
\end{aligned}
\]
where we first used commutativity of the trace, and then the fact (following from the trace H\"older inequality) that
$|\tr(\A \bB)| \le \|\bB\|  \tr(\A)$ for~$\A \succq 0$ and~$\bB$ with compatible dimensions. The claim follows.
\end{proof}

%% file: proof-stability.tex
\section{Proof of Lemma~\ref{lem:stability}}
\label{sec:proof-stability}

$\boldsymbol{1^o.}$
We start by deriving the consequences of~\eqref{eq:stability-premise}. First, note that
\begin{equation*}
\S_{\lam}^{1/2}\wh \S_{\lam}^{-1}\S_{\lam}^{1/2} = \S_{\lam}^{1/2}\left[\S_{\lam} - (\S -\wh\S)\right]^{-1}\S_{\lam}^{1/2} = \left[\Id - \S_{\lam}^{-1/2}(\S -\wh\S)\S_{\lam}^{-1/2}\right]^{-1}.
\end{equation*}
Whence, using~\eqref{eq:stability-premise} and the similarity rules,
\begin{align}
\label{eq:replace-psd}
\frac{2}{3} \Id \prccq \S_{\lam}^{1/2} \wh \S_{\lam}^{-1} \S_{\lam}^{1/2} \prccq 2 \Id.
\end{align}
By the properties of the spectral norm, this implies
\begin{equation}
\label{eq:replace-under-norm}
\begin{aligned}
\left\|\S_{\lam}^{1/2} \wh \S_{\lam}^{-1/2} \right\|^2 = \left\|\wh \S_{\lam}^{-1/2} \S_{\lam}^{1/2} \right\|^2 &\le 2; \\
\left\|\S_{\lam}^{-1/2} \wh \S_{\lam}^{1/2} \right\|^2 = \left\|\wh \S_{\lam}^{1/2} \S_{\lam}^{-1/2} \right\|^2 &\le \frac{3}{2}.
\end{aligned}
\end{equation}
Using that, and proceding as in the proof of Lemma~\ref{lem:df}, we can bound the degrees of freedom surrogate~$\tr(\S \wh \S^{-1}_{\lam})$ in terms of the true quantity~$\df_{\lam}(\S) = \tr(\S \S_{\lam}^{-1})$:
\begin{align*}
\tr(\S \wh \S^{-1}_{\lam}) - \df_{\lam}(\S)
&= \tr\left[\S (\S_{\lam}^{-1} - \wh \S_{\lam}^{-1}) \right]\\ 
&= \tr\left[ \S \S_{\lam}^{-1/2}\left(\Id - \S_{\lam}^{1/2}\wh \S_{\lam}^{-1}\S_{\lam}^{1/2}\right) \S_{\lam}^{-1/2} \right]\\
&= \tr\left[\S_{\lam}^{-1/2} \S \S_{\lam}^{-1/2} \left(\Id - \S_{\lam}^{1/2}\wh \S_{\lam}^{-1}\S_{\lam}^{1/2}\right) \right],
\end{align*}
where in the third line we used commutativity of the trace. Applying the trace H\"older inequality as in the proof of Lemma~\ref{lem:df}, we obtain
\begin{align}
\label{eq:df-surrogate-bound}
\left|\tr(\S \wh \S^{-1}_{\lam}) - \df_{\lam}(\S) \right| \le \left\|\Id - \S_{\lam}^{1/2}\wh \S_{\lam}^{-1}\S_{\lam}^{1/2}\right\| \df_{\lam}(\S) \le 3 \df_{\lam}(\S),
\end{align}
where we combined the triangle inequality with the right-hand side of~\eqref{eq:replace-psd}. 

$\boldsymbol{2^o.}$
We now invoke the results of~\cite{wei2017estimation} (in what follows, the expectation is conditioned on~$\wh\S_{\lam}$). 
Note that conditionally on~$\wh\S$, the random vectors~$Z_j = \wh\S_{\lam}^{-1/2}X_j$ are i.i.d.~with mean zero and covariance~$\wh\bJ = \wh\S_{\lam}^{-1/2} \S \wh\S_{\lam}^{-1/2}.$ 
By~\eqref{eq:variance-bound}, and using the linear invariance of~\eqref{ass:kurtosis}, 
\[
\left\|\E\left[\|Z_j\|^2 Z_j \otimes Z_j \right] \right\| \le \k^4 \|\wh\bJ\|^2 \reff(\wh\bJ) = \k^4 \|\wh\bJ\| \tr(\wh\bJ).
\]
Using~\eqref{eq:replace-psd} and~\eqref{eq:df-surrogate-bound} to bound~$\|\wh\bJ \|$ and~$\tr(\wh\bJ)$ correspondingly, this results in
\begin{equation}
\label{eq:stability-moment-bound}
\|\E \|Z_j\|^2 Z_j \otimes Z_j \| \le 8 \k^4 \df_{\lam}(\S).
\end{equation}
On the other hand, the estimator~$\wh\S^{(+)}$ defined in~\eqref{eq:stability-step} satisfies
\[
\wh\S_{\lam}^{-1/2}\wh\S^{(+)}\wh\S_{\lam}^{-1/2} = \frac{1}{m} \sum_{j=1}^m \rho_{\theta}(\|Z_j\|) Z_j \otimes Z_j,
\]
that is,~$\wh\S_{\lam}^{-1/2}\wh\S^{(+)}\wh\S_{\lam}^{-1/2}$ is precisely the Wei-Misnker estimator (cf.~\eqref{def:minsker}) of~$\wh \bJ$, computed from the sample~$(Z_1, ..., Z_m)$. 
Hence, combining the result of Theorem~\ref{th:minsker} with~\eqref{eq:stability-moment-bound}, we see that whenever
\[
\theta \ge 2\sqrt{2}\k^2 \sqrt{\frac{m \df_{\lam}(\S)}{\log(2d/\delta)}},
\]
with conditional probability at least~$1-\delta$ it holds
\[
\left\| \wh\S_{\lam}^{-1/2} (\wh\S^{(+)} - \S ) \wh\S_{\lam}^{-1/2} \right\| \le \frac{2\theta \log(2d/\delta)}{m}.
\]
Finally, we arrive at~\eqref{eq:stability-result} by writing
\[
\begin{aligned}
\left\|\S_{\lam/2}^{-1/2}(\wh\S^{(+)}-\S )\S_{\lam/2}^{-1/2}\right\| 
&\le \left\|\S_{\lam/2}^{-1/2} \S_{\lam}^{1/2} \right\|^2 \left\|\S_{\lam}^{-1/2} \wh\S_{\lam}^{1/2} \right\|^2  \left\| \wh\S_{\lam}^{-1/2} (\wh\S^{(+)} - \S ) \wh\S_{\lam}^{-1/2} \right\|,
\end{aligned}
\]
noting that~$\left\|\S_{\lam/2}^{-1/2} \S_{\lam}^{1/2} \right\|^2 = \left\|\S_{\lam/2}^{-1} \S_{\lam} \right\| \le 2$, and bounding~$\left\|\S_{\lam}^{-1/2} \wh\S_{\lam}^{1/2}\right\|^2 \le 3/2$ via~\eqref{eq:replace-under-norm}.
\hfill \qed

%% file: proof-lepski.tex
\section{Proof of Theorem~\ref{th:calibrated-adaptive} and Corollary~\ref{cor:lepski-simple}}
\label{sec:proof-lepski}

Let us call the truncation level~$\theta_j = \thetamin 2^j$, with~$j \in \cJ$,~\emph{admissible} if it satisfies the condition in~\eqref{eq:lepski-choice}, so that~$\theta_{\whj}$ is the smallest such level.
Let~$j^* \in \cJ$ be the minimal~$j \in \cJ$ such that~$\theta_{j^*} \ge \theta^*$ for~$\theta_*$ defined in~\cref{eq:calibrated-basic-threshold}; note that this is always possible by the definition~\eqref{eq:lepski-grid}, and we have
\begin{equation}
\label{eq:lepski-grid-twice}
\theta_{j^*} \le 2 \theta_*.
\end{equation}

\proofpoint{1}
Let us prove that~$\theta_{j^*}$ is admissible with probability at least~$1-\delta$. 
Indeed, due to~\eqref{eq:calibrated-lepski-min-sample-size-thetamax}, the premise~\eqref{eq:calibrated-basic-min-sample-size} of Theorem~\ref{th:calibrated-basic} holds for any~$\theta = \theta_j$ with~$j \in \cJ$ (recall that~$\theta_j \le 2\thetamax$). 
On the other hand, the premise~\eqref{eq:calibrated-basic-threshold} of Theorem~\ref{th:calibrated-basic} holds whenever~$\theta_j \ge \theta_{*}$.
Hence the bound~\eqref{eq:calibrated-basic-result} of Theorem~\ref{th:calibrated-basic} holds for all~$\theta = \theta_{j}$ with~$j \ge j^*$, and by the union bound we get that with probability at least~$1-\delta,$
\begin{equation}
\label{eq:calibrated-basic-result-for-lepski}
\left\|\S_\lam^{-1/2} ( \wh\S_j - \S ) \S_\lam^{-1/2}\right\| \le  \veps_j = \frac{24\theta_j \sqrt{\q\log(4\q d|\cJ|/\delta)\log(4d|\cJ|/\delta)}}{n}, \quad \forall j \ge j_*, \;\; j \in \cJ.
\end{equation}
Moreover, from~\eqref{eq:calibrated-lepski-min-sample-size-thetamax} we also obtain
\[
\veps_{j} \le \frac{1}{2} \sqrt{\frac{\log(4d|\cJ|/\delta)}{\q \log(4\q d|\cJ|/\delta)}} \le \frac{1}{2}, \quad j \in \cJ.
\]
Whence for any~$j \in \cJ$ such that~$j \ge j_*$ we have, under the event~\eqref{eq:calibrated-basic-result-for-lepski}, and denoting~$\wh\S_{j,\lam} = \wh\S_{j} + \lam \Id$,
\begin{equation*}
\begin{aligned}
\left\| \wh\S_{j,\la}^{-1/2} (\wh\S_{j} - \wh\S_{j_*}) \wh\S_{j,\la}^{-1/2} \right\| 
&\le \left\|\wh\S_{j,\la}^{-1/2} \S_{\la}^{1/2} \right\|^2 \cdot \left\|\S_{\la}^{-1/2} (\wh\S_{j} - \wh\S_{j_*}) \S_{\la}^{-1/2} \right\| \\
&\le \left\|\wh\S_{j,\la}^{-1/2} \S_{\la}^{1/2} \right\|^2  \left(\left\|\S_{\la}^{-1/2} (\wh\S_{j} - \S) \S_{\la}^{-1/2} \right\| + \left\|\S_{\la}^{-1/2} (\wh\S_{j_*} - \S) \S_{\la}^{-1/2} \right\| \right) \\
&\le 2(\veps_j + \veps_{j_*})
\end{aligned}
\end{equation*}
where we used~\eqref{eq:calibrated-basic-result-for-lepski} to bound the terms in the parentheses, and also used (cf.~\eqref{eq:replace-under-norm} in Appendix~\ref{sec:proof-stability}):
\[
\left\|\wh\S_{j,\la}^{-1/2} \S_{\la}^{1/2} \right\|^2 \le \frac{1}{1-\veps_{j}} \le 2, \quad \forall j \ge  j_*.
\]
Thus,~$j_*$ is indeed admissible with probability~$\ge 1-\delta$.

\proofpoint{2}
Whenever~$j_*$ is admissible, we have~$\whj \le j_*$, whence~$\veps_{\whj} \le \veps_{j_*}$ using that~$\veps_j$ increases in~$j$. Thus, with probability at least~$1-\delta$ it holds 
\begin{equation}
\label{eq:lepski-final-chain}
\begin{aligned}
&\left\|\S_\lam^{-1/2} ( \wh\S_{\whj} - \S ) \S_\lam^{-1/2}\right\| \\
&\quad \le \left\|\S_{\la}^{-1/2} \wh\S_{j_*,\la}^{1/2}\right\|^2 \cdot \left\|\wh\S_{j_*,\lam}^{-1/2} ( \wh\S_{\whj} - \S ) \wh\S_{j_*,\lam}^{-1/2}\right\| \\
&\quad\le \left\|\S_{\la}^{-1/2} \wh\S_{j_*,\la}^{1/2}\right\|^2 \cdot \left( \left\|\wh\S_{j_*,\lam}^{-1/2} ( \wh\S_{j_*} - \wh\S_{\whj} ) \wh\S_{j_*,\lam}^{-1/2}\right\| + \left\|\wh\S_{j_*,\lam}^{-1/2} ( \wh\S_{j_*} -  \S ) \wh\S_{j_*,\lam}^{-1/2}\right\|\right) \\
&\quad\le \frac{3}{2} (2(\veps_{\whj} + \veps_{j_*}) + \veps_{j_*}) \le \frac{15}{2} \veps_{j_*},
\end{aligned}
\end{equation}
where in order to obtain the last line we used (cf.~\eqref{eq:replace-under-norm} in Appendix~\ref{sec:proof-stability}) that
\[
\left\|\S_{\la}^{-1/2} \wh\S_{j_*,\la}^{1/2}\right\|^2 \le 1 + \veps_{j_*} \le {3}/{2}.
\]
Finally, combining this with the expression for~$\theta_*$ in~\eqref{eq:calibrated-basic-threshold}, and using~\eqref{eq:lepski-grid-twice}--\eqref{eq:lepski-final-chain}, we arrive at the claimed bound. 
\hfill \qed

\paragraph{Proof of Corollary~\ref{cor:lepski-simple}.}
First,~$\thetamax$ defined in the premise satisfies~\cref{eq:calibrated-lepski-min-sample-size-thetamax} by construction; moreover,~\cref{eq:calibrated-lepski-min-sample-size-thetamax} is satisfied as an equality.
On the other hand, by simple algebra~\cref{eq:lepski-simple-condition} guarantees that~$\thetamax \ge \theta_*$ for~$\theta_*$ defined in~\cref{eq:calibrated-basic-threshold}.
Finally, verifying that~$\thetamin \le \theta_*$ is trivial using that~$\kappa \ge 1$ and~$\df_{\lam}(\S) \ge 1$.
\hfill \qed

%% file: proof-ridge.tex
\section{Proof of Theorem~\ref{th:ridge-robust}}
\label{sec:proof-ridge}

\proofpoint{1}
First of all, note that~$n$ satisfying~\eqref{eq:calibrated-basic-min-sample-size-simplified} suffices to guarantee that
\begin{equation}
\label{eq:ridge-norm-equivalence}
\left\|\S_\lam^{-1/2} ( \wh\S - \S ) \S_\lam^{-1/2}\right\| \le 48\k^2\sqrt{\frac{\df_{\lam}(\S)\log(2d/\delta)}{n}} \le \frac{1}{2}
\end{equation}
holds with probability at least~$1-\delta/2$, cf.~\eqref{eq:calibrated-basic-result-simplified}. 
Note also that~$\wh \S_{\lam}$ is independent from~$(X_1, ..., X_n)$, hence the vectors~$\wh Z_i = \wh\S_{\lam}^{-1/2} X_i Y_i$,~$1 \le i \le n$, are independent when conditioned on~$(X_{n+1}, ..., X_{2n})$. Finally, the conditional to the hold-out sample~$X_{n+1}, ..., X_{2n}$ expectation of~$\wh Z_i$ is
\begin{equation}
\label{eq:cond-expectation}
\wh \E[\wh Z_i] = \wh \S_{\lam}^{-1/2} \S w_*,
\end{equation}
where we used that the residual~$\xi = Y - X^{\top} w_*$ satisfies~$\E[\xi X] = 0$, which follows from the fact that~$w_*$ minimizes~$L(w)$. 

\proofpoint{2}
We now decompose the excess risk of~$\bar w_{\lam}$ as follows: 
\begin{equation}
\label{eq:ridge-decomposition}
\begin{aligned}
\sqrt{L(\bar w_{\lam}) - L(w^*)}  
\le \underbrace{\| \S^{1/2}(\bar w_{\lam} - \wh w_{\lam}) \|}_{E_1} + \underbrace{\| \S^{1/2}(\wh w_{\lam} - w_{\lam}) \|}_{E_{2}} + \underbrace{\| \S^{1/2}(w_{\lam} - w^*) \|}_{E_3},
\end{aligned}
\end{equation}
where~$w_{\lam}$, given by
\begin{equation}
\label{eq:ridge-explicit}
w_{\lam} = \S_{\lam}^{-1} \S w^*,
\end{equation}
is the minimizer of~$L_{\lam}(w) = L(w) + \lam \|w\|^2$, and~$\wh w_{\lam} := \wh\E[\wh \S_{\lam}^{-1/2} \wh Z_1]$ can be calculated using~\eqref{eq:cond-expectation}:
\begin{equation}
\label{eq:ridge-conditional-estimator}
\wh w_{\lam} = \wh\S_{\lam}^{-1} \S w^*.
\end{equation}
The easiest to control in~\eqref{eq:ridge-decomposition} is the term~$E_{3}$ corresponding to the squared bias in the fixed-design setting:
\begin{equation}
\label{eq:ridge-bias-bound}
\| \S^{1/2}(w_{\lam} - w^*) \|
\le \| \S_{\lam}^{1/2}(w_{\lam} - w^*) \| = \lam \|\S_{\lam}^{-1/2} w^*\|, 
\end{equation}
resulting in the second term in the brackets in~\eqref{eq:ridge-result}.

\proofpoint{3}
On the other hand, using~\eqref{eq:ridge-explicit}--\eqref{eq:ridge-conditional-estimator} we have
\begin{equation}
\label{eq:ramdom-design-bound}
\begin{aligned}
E_2 
&\le \left\| \S_{\lam}^{1/2}(\S_{\lam}^{-1} - \wh \S_{\lam}^{-1}) \S w_* \right\| \\
&= \left\| \S_{\lam}^{1/2} \wh\S_{\lam}^{-1} (\S - \wh \S) \S_{\lam}^{-1} \S w_* \right\| \\
&\le \left\| \S_{\lam}^{1/2} \wh\S_{\lam}^{-1/2} \right\|^2 \cdot \left\| \S_{\lam}^{-1/2} (\S - \wh \S) \S_{\lam}^{-1/2} \right\| \cdot \left\|\S_{\lam}^{-1/2} \S^{1/2} \right\| \cdot \left\|\S^{1/2} w^* \right\|,
\end{aligned}
\end{equation}
where the last inequality can be verified by removing the norms. 
Under the event~\eqref{eq:ridge-norm-equivalence}, we can bound the first term by a constant (see the proof of Lemma~\ref{sec:proof-stability} in Appendix~\ref{sec:proof-stability}), and the second term by
\[
O(1)\k^2\sqrt{\frac{\df_{\lam}(\S)\log(2d/\delta)}{n}},
\] 
cf.~\eqref{eq:ridge-norm-equivalence}. The third term is at most one. Finally, we have~$\|\S^{1/2} w^*\|^2 = \E[ (X^\top w^*)^2 ]\le \E[Y^2] = v^2$. Collecting the above, under the event~\eqref{eq:ridge-norm-equivalence} we have
\[
E_2 \le O(1) \kappa^2 \sqrt{\frac{v^2\df_{\lam}(\S)\log(2d/\delta)}{n}}.
\]

\proofpoint{4}
Finally, let us estimate the term~$E_1$ which corresponds to the additive noise, and delivers the first term in the brackets in~\eqref{eq:ridge-result}. Note that we can bound
\[
\begin{aligned}
E_1 
&= \| \S^{1/2}(\bar w_{\lam} - \wh w_{\lam}) \| \\
&\le \| \S_{\lam}^{1/2}(\bar w_{\lam} - \wh w_{\lam}) \| \\
&\le \left\|\S_{\lam}^{1/2} \wh\S_{\lam}^{-1/2} \right\|  \cdot \left\| \wh\S_{\lam}^{1/2}(\bar w_{\lam} - \wh w_{\lam}) \right\| \\
&= \left\|\S_{\lam}^{1/2} \wh\S_{\lam}^{-1/2} \right\|  \cdot \left\| \bar Z - \wh\E[\wh Z]\right\|, \;\; \text{where} \;\; \wh Z = \wh \S_{\lam}^{-1/2} X Y,
\end{aligned}
\]
cf.~\eqref{eq:ridge-estimator-robust} and~\eqref{eq:cond-expectation}. Recall that under the event~\eqref{eq:ridge-norm-equivalence}, the first term in the product is bounded by a constant, and it remains to control the deviations of the estimator~$\bar Z$ of~$\wh Z$ from the (conditional) average~$\wh\E[\wh Z]$. 
To this end, consider the following construction due to~\cite[Sec.~3.3]{minsker2018sub}. 
For the general matrix~$A \in \R^{d_1 \times d_2}$, define its Hermitian dilation
\begin{equation}
\label{eq:hermitian-dilation}
\cH(A) = 
\left(
\begin{matrix} 
0_{d_1 \times d_1} & A \\
A^\top & 0_{d_2 \times d_2}
\end{matrix}\
\right),
\end{equation}
and for any~$\vphi: \R \to \R$, define the map on the space of~$(d_1 + d_2) \times (d_1 + d_2)$ Hermitian matrices:
\[
\vphi(\A) := Q \vphi(\bLambda) Q^\top = Q \, \textbf{diag}(\vphi(\lam_1) \cdots \vphi(\lam_d)) Q^\top,
\]
where~$Q \boldsymbol{\Lambda} Q^\top$ is the eigendecomposition of~$\A$. 
In this notation, consider the following estimator of~$\E[A] \in \R^{d_1 \times d_2}$ from i.i.d.~copies~$A_1, ..., A_n$ of~$A$: compute the~$(d_1 + d_2) \times (d_1 + d_2)$ Hermitian matrix
\begin{equation}
\label{eq:minsker-rectangular}
\wh \bT = \frac{1}{n} \sum_{i=1}^n \psi_{\bar\theta}(\cH(A_i)), 
\end{equation}
where~$\psi_{\bar \theta}(\cdot)$ is the matrix map corresponding to~\eqref{eq:truncation-map} with truncation level~$\bar\theta$, and then output the top right block of~$\wh \bT$ (i.e., the one corresponding to~$A$ in~$\cH(A)$) as the final estimate. 
As proved in~\cite[Cor.~3.1]{minsker2018sub}, the resulting estimate satisfies
\[
\left\| \bar A - \E[A]\right\| \le O(1)\frac{\bar\theta \log(1/\delta)}{n}
\]
with probability at least~$1-\delta$, provided that~$\delta \le 1/2$, and
\[
\bar \theta = \sqrt{\frac{n \overline{w}}{\log(1/\delta)}} \;\;  \text{for some} \;\; \overline{w} \ge \tr \E[A \otimes A].
\]

On the other hand, one can verify that this construction reduces to~$\bar Z$ when estimating~$\wh\E[\wh Z]$ from~$\wh Z_1, ..., \wh Z_n$ with the same~$\bar\theta$.
Thus, with (conditional) probability~$\ge 1-\delta/2$ it holds
\[
\left\| \bar Z - \wh\E[\wh Z]\right\| \le O(1)\frac{\bar\theta \log(2/\delta)}{n}
\]
whenever~$\bar\theta$ is taken to be
\[
\bar \theta = \sqrt{\frac{n \overline{w}}{\log(1/\delta)}} \;\;  \text{for some} \;\; \overline{w} \ge \tr \wh \E[\wh Z \otimes \wh Z],
\]
which then results in the bound
\[
\left\| \bar Z - \wh\E[\wh Z]\right\| \le O(1)\sqrt{\frac{\overline{w} \log(2/\delta)}{n}}.
\]
It remains to bound~$\bar w = \tr \wh \E[\wh Z \otimes \wh Z]$. 
Using the trace H\"older inequality, we have
\[
\begin{aligned}
\tr \left[\wh \E[\wh Z \otimes \wh Z] \right]
&= \tr\left[\wh \S_{\lam}^{-1} \E[Y^2 X X^\top] \right]  \\
&\le \left\|\S_{\lam}^{1/2} \wh\S_{\lam}^{-1/2} \right\| \cdot \tr\left[\S_{\lam}^{-1} \E[Y^2 X X^\top] \right], \\
&= \left\|\S_{\lam}^{1/2} \wh\S_{\lam}^{-1/2} \right\| \cdot \E\left[\left\|Y \S_{\lam}^{-1/2}X \right\|^2\right],
\end{aligned}
\]
where the first term on the right is at most a constant under~\eqref{eq:ridge-norm-equivalence}.
Finally, under the fourth-moment assumptions in the premise of the theorem, we can bound the last term coordinatewise, using that each coordinate of~$\S_{\lam}^{-1/2}X$ is simply the projection of~$\S_{\lam}^{-1/2}X$ onto the corresponding coordinate vector, and proceeding via Cauchy-Schwarz:
\[
\E\left[\left\|Y \S_{\lam}^{-1/2}X \right\|^2\right] \le \kresp^2 \kappa^2 v^2 \E\left[\left\|\S_{\lam}^{-1/2}X \right\|^2\right] = \kresp^2 \kappa^2 v^2 \df_{\lam}(\S).
\]
Combining the previous steps, we obtain the claimed result. 
\hfill \qed